\documentclass[12pt]{amsart}
\usepackage[pdftex,pagebackref,letterpaper=true,colorlinks=true,pdfpagemode=none,urlcolor=blue,linkcolor=blue,citecolor=blue,pdfstartview=FitH]{hyperref}
\usepackage{tikz}
\usepackage{amsmath}
\usepackage{amssymb}
\usepackage{amsthm}
\newtheorem{theorem}{Theorem}[section]

\newtheorem{lemma}[theorem]{Lemma}
\newtheorem{proposition}[theorem]{Proposition}
\newtheorem{definition}[theorem]{Definition}
\newtheorem{remark}[theorem]{Remark}
\usepackage{enumitem}
\usepackage{gensymb}
\usepackage{siunitx}
\usepackage[mathscr]{euscript}
\usepackage{physics}

\usepackage[colorlinks=true]{hyperref}     % creates hyperlinks in the pdf

\usepackage{caption}

\usepackage{cite}
\usepackage{amsrefs}
\newcommand{\R}{\mathbb{R}}
\newcommand{\N}{\mathbb{N}}

\newcommand{\p}{\partial}
\newcommand{\eps}{\varepsilon}

\usepackage{subfiles}
\numberwithin{equation}{section}
\usepackage{tcolorbox}
\usepackage{MnSymbol}
\title[On the viscosity linearization method]{On the viscosity linearization method without compactness}
\author{Anthony Salib}
\address{Department of Mathematics, University of Duisburg-Essen, Thea-Leymann-Strasse 9, 45127 Essen,Germany.}
\email{anthony.salib@stud.uni-due.de}
\begin{document}
\renewcommand{\div}{\textnormal{div}}
\newcommand{\dist}{\textnormal{dist}}
\newcommand{\loc}{\textnormal{loc}}
\newcommand{\osc}{\textnormal{osc}\hspace{0.5mm}}
\newcommand{\diam}{\textnormal{diam}}
\newcommand{\naught}{_{\circ}}
\newcommand{\supp}{\textnormal{supp}}

\maketitle
\begin{abstract}
	Savin's small perturbation approach has had far reaching applications in the theory of non-linear elliptic and parabolic PDE. In this short note, we revisit his seminal proof of De-Giorgi's improvement of flatness theorem for minimal surfaces and provide an approach based on the Harnack inequality that avoids the use of compactness arguments.
\end{abstract}

\section{Introduction}
In this paper we consider minimizers of the perimeter functional in $B_1$
\begin{equation*}
	P(E) = \int_{B_1} \abs{D\chi_{E}} := \sup_{\substack{g\in C^1_0(B_1;\R^n)\\ \norm{g}_{L^{\infty}}\leq 1}}\int_{E} \div(g).
\end{equation*}
We recall that a set $E \subset B_1$ is said to be minimal in $B_1$ if for any set $F \subset B_1$ such that $(E \backslash F) \cup (F\backslash E) \subset\subset B_1$ it holds that $P(E)\leq P(F)$. Moreover, the boundary of a minimal set $E$, $\p E$, is called a minimal surface. 

By viewing $\p E$ as a multi-valued graph in $\R^{n-1}$, it is shown in \cite{caffarelli1993elementary} that minimal surfaces are viscosity solutions (see Section \ref{preliminaries} for a precise statement) of the minimal surface equation
\begin{equation}
	(1+\abs{Du}^2)\Delta u - Du^{T} D^2 u Du= 0.
	\label{MSE}
\end{equation}

In the seminal work \cite{savin2009phase}, Savin showed that minimal surfaces in this setting satisfy the following Harnack inequality. 
\begin{theorem}[Harnack Inequality]\label{harnack}
	There exists universal constants $\eps_1(n)$ and $\eta > 0$  such that if $\p E$ is a minimal surface and 
	\begin{equation*}
		\p E \cap B_1 \subset \left\{\abs{x_n} \leq \eps\right\}
	\end{equation*}
	for some $\eps \leq \eps_1(n)$, then 
	\begin{equation*}
		\p E \cap B_{1/2} \subset \left\{\abs{x_n} \leq \eps(1-\eta)\right\}.
	\end{equation*}
\end{theorem}

Using Theorem \ref{harnack} and a compactness argument, Savin in \cite[Theorem 5.2]{savin2009phase} then showed that there exists universal constants $\eps\naught >0$ and $r\naught > 0$ such that if
\begin{equation*}
	\p E \cap B_1 \subset \left\{\abs{x_n} \leq \eps\right\} 
\end{equation*}
for some $\eps \leq \eps\naught$, then there exists $\nu \in \mathbb{S}^{n-1}$ such that
\begin{equation*}
	 \p E \cap B_{r\naught} \subset \left\{\abs{x\cdot\nu}\leq \frac{\eps}{2}r\naught\right\}.
\end{equation*}
This is the celebrated improvement of flatness theorem for minimal surfaces and is a key step in the corresponding regularity theory. In fact, by iterating this result one readily obtains the smoothness of a minimal surface around the origin. 

This improvement of flatness acts as a prototype for a large number of ``$\eps$-regularity results" - if a solution $u$ of some PDE  is close enough to a special solution, then $u$ is smooth. In the case of minimal surfaces, the special solution considered is just the hyperplane. 

The key idea in establishing such an $\eps$-regularity result is to first consider solutions to the linearized equation around the special solution. For instance, if the PDE is of the form $F(D^2u,Du)=0$ (as is \eqref{MSE}), then a function $v$ satisfies the linearized equation around a solution $w$ if $L_w(v) = 0$, where 
\begin{equation}
	L_w(v) = \lim_{\eps \to 0} \frac{F(D^2(w+\eps v),D(w+\eps v))-F(D^2w,Dw)}{\eps}.
\end{equation}
The next step is to then show that if $u$ is close enough to the special solution, then certain properties of solutions of the linearized equation are inherited by $u$.

This idea has become ubiquitous in the regularity theory of non-linear elliptic PDE \cite{savin2007small,de2011free,de2015free,de2021perturbative,fernandezserra2020regularity,audrito2022interface}. For instance, in the same seminal paper \cite{savin2009phase}, Savin uses the linearization method to obtain an improvement of flatness theorem for certain solutions of the Allen-Cahn equation (which he then uses to conclude the 1D-symmetry of such solutions). In the case of the Bernoulli free boundary problem, an analogous improvement of flatness theorem around free boundary points is obtained using these same ideas \cite{de2011free}. Moreover, this method has been extended to the parabolic setting in the one-phase Stefan problem \cite{de2021perturbative}. 

In the case of the minimal surface equation \eqref{MSE}, the linearized PDE is simply $\Delta u =0$ and so one expects that minimal surfaces are well approximated by harmonic functions. Although Savin used this linearization technique with a viscosity approach in \cite{savin2009phase}, the improvement of flatness theorem was first shown by De-Giorgi in the variational setting \cite{de1961frontiere}. In both \cite{savin2009phase} and \cite{de1961frontiere} a compactness argument is used to obtain a harmonic function that appropriately approximates the minimal surface and as a result, the value $\eps\naught$ appearing in their results is not quantified. Schoen and Simons in \cite{schoen1982new} provided a constructive proof in the variational setting, that is, a proof in which all the constants can be explicitly computed. Moreover, Caffarelli and C\'ordoba in \cite{caffarelli1993elementary} give another constructive proof using viscosity techniques.

While also based on viscosity methods, the approach given in \cite{savin2009phase} is distinct to that in \cite{caffarelli1993elementary} as it is based on a Harnack inequality for minimal surfaces (Theorem \ref{harnack}). Broadly speaking, the framework established by \cite{savin2009phase} is that a Harnack inequality implies the improvement of flatness result. Similar to the proof in \cite{savin2009phase}, the previously mentioned applications of the viscosity linearization method utilise a compactness argument to obtain solutions to the linearized PDE. For physical applications, it is important to quantify the constants appearing in the theory as it reveals the scales at which these physical phenomena are expected to conform to the theory's predictions. The objective of this note is to provide a constructive proof of the improvement of flatness theorem within the framework of \cite{savin2009phase}. We believe the methods developed in this note can be applied to other settings in order to obtain quantitative statements analogous to our main result, Theorem \ref{improvementofflatness} below. 

Recalling the values $\eps_1$ and $\eta$ from Theorem \ref{harnack}, our main result can be stated as follows.

\begin{theorem}[Improvement of Flatness]
	\label{improvementofflatness}
	Let $\alpha = -\log(1-\eta)$ and $\gamma = \frac{1}{1-\alpha}$ and define the constants
	\begin{equation*}
		r\naught = \frac{1}{2^{16}n^2}, \hspace{2mm} \eps\naught = \left(\frac{\eps_1^{\gamma\alpha}}{2^{33+5\alpha}n^3}\right)^{\frac{8}{\gamma\alpha^2}}.
	\end{equation*}
	If $\p E$ is a minimal surface in $B_1$, $0\in \p E$ and
	$$\p E \cap B_1 \subset \left\{\abs{x_n}\leq\eps \right\}$$ for some $\eps\leq\eps\naught$, then there exists $\nu \in \mathbb{S}^{n-1}$ such that $$\p E \cap B_{r\naught} \subset \left\{\abs{x\cdot \nu}\leq\frac{\eps}{2}r\naught\right\}.$$
\end{theorem}
We note that Theorem \ref{harnack} was already proven in \cite{savin2009phase} using constructive methods and so the constants $\eps_1$ and $\eta$ appearing in the statement of Theorem \ref{improvementofflatness} are known. Moreover, we refer to \cite{de2023short} for a short proof of Theorem \ref{harnack} in which the constants can be easily quantified. Therefore, in this note we will only be concerned with obtaining estimates for $r\naught$ and $\eps\naught$ in terms of the constants $\eps_1$ and $\eta$. 
\subsection{Structure of the proof}
The main step in the proof of Theorem \ref{improvementofflatness} is to show that the minimal surface is well approximated by a harmonic function. Of course since $\p E$ is not assumed to be a graph in the $x_n$-direction, we must first appropriately regularize $\p E$ using inf-convolution (c.f. Lemma \ref{holdergraph}). We will then show that the harmonic replacement of this regularization is close in $L^{\infty}$ to the surface itself (see Proposition \ref{closetoharmonic}). In this way we explicitly construct the desired harmonic approximation. Once this approximation is obtained, we can conclude similarly to \cite{savin2009phase}. 
\subsection{Acknowledgements}
The results described in this paper were achieved in the Master's thesis of the author. The author expresses his gratitude to his advisor Joaquim Serra for his patient guidance and support during the preparation of this result, as well as to Federico Franceschini for helpful comments on a preliminary version of the manuscript. The author is supported by the Universität Duisburg-Essen Faculty of Mathematics scholarship.

\section{Notation and Preliminaries}\label{preliminaries}
\subsection{Notation}
Throughout this work $\R^n$ will be endowed with the Euclidean inner product $x\cdot y$ and it's induced norm $\abs{x}$. For any matrix $Q \in \R^{n\times n}$ we will denote by $\abs{Q}$ the induced norm 
\begin{equation*}
 \abs{Q} = \max_{\abs{x}\leq 1} \abs{Qx}.
\end{equation*}
$B_r(x)$ represents the ball of radius of $r$ with centre $x$. We will also represent points $x\in \R^n$ as $(x',x_n) \in \R^{n-1}\times \R$ and use $B'_r(x')$ to represent a ball in $\R^{n-1}$ with centre $x'$ and radius $r$. Moreover, given any measurable set $E \subset \R^n$ we will denote by $E^c$ the set $\R^n\backslash E$. 

We will denote by $2^{\R}$ all possible subsets of $\R$. For a multi-valued graph $A: \R^{n-1} \to 2^{\R}$, we define
\begin{equation*}
	A(x')-A(y') = \sup\left\{t-s; t\in A(x'), s \in A(y') \right\},
\end{equation*}
and
\begin{equation*}
	\abs{A}=\sup_{x'\in\R^{n-1}}\sup\left\{s; s\in A(x') \right\}.
\end{equation*}
Note that the oscillation of $A$ in any $\Omega \subset \R^{n-1}$ is well defined as 
\begin{equation*}
	\underset{\Omega}{\osc} A = \sup_{\Omega} A - \inf_{\Omega} A.
\end{equation*}

Finally, given any $x \in \R$, $\lfloor x \rfloor$ represents the number $n \in \mathbb{Z}$ such that $n \leq x \leq n+1$. 
\subsection{Preliminaries}
We will follow the notation and definitions given in \cite{savin2009phase}. 

\begin{definition}\label{def:element}
	Given a measurable set $E \subset \R^n$ we say that $x \in \p E$ if for any $r>0$ there holds that $\abs{B_r(x) \cap E} > 0$ and $\abs{B_r(x) \cap E^c} >0$. 
\end{definition}

Recalling the density estimates for minimal surfaces (\cite[Theorem 4.1]{savin2009phase}) it is obvious that $\p E$ is closed in the sense of Definition \ref{def:element}

We will now recall the definition of viscosity solutions for \eqref{MSE} (see \cite[Section 5]{savin2009phase}). Given any smooth function $\varphi : \R^{n-1} \to \R$, we define it's sub-graph to be the set $\Gamma_{\varphi}=\{x' \in \R^{n-1}: x_n < \varphi(x')\}$. Similarly it's super-graph is the set $\Gamma^{\varphi} :=  \{x': x_n > \varphi(x')\} $.

\begin{definition}[Viscosity solution]\label{definition:viscosity:solution}
A set $E \subset B_1$ is a viscosity solution of \eqref{MSE} if for any smooth function $\varphi : \R^{n-1} \to \R$ such that $\Gamma_{\varphi} \cap B_{r}(y)$ is contained in either $E$ or $E^c$ for some $y \in \p E \cap \p \Gamma_{\varphi}$ and $r>0$, then 
\begin{equation*}
	(1+\abs{D\varphi}^2)\Delta \varphi - D\varphi^{T} D^2 \varphi D\varphi\leq0.
\end{equation*}
If $\Gamma^{\varphi} \cap B_{r}(y)$ is contained in either $E$ or $E^c$ for some $y \in \p E \cap \p \Gamma^{\varphi}$ and $r>0$, then 
\begin{equation*}
	(1+\abs{D\varphi}^2)\Delta \varphi - D\varphi^{T} D^2 \varphi D\varphi\geq0.
\end{equation*}
\end{definition}

We have the following result from \cite{caffarelli1993elementary}.
\begin{theorem}
	Let $\p E$ be a minimal surface in $B_1$. Then $\p E$ is a viscosity solution of \eqref{MSE} in the sense of Definition \ref{definition:viscosity:solution}.
\end{theorem}

A remark on the constants appearing in the Harnack inequality, Theorem \ref{harnack}, is in order. 
\begin{remark}\label{remark:assumption}
	A careful computation of the constants $\eps_1(n)$ and $\eta$ appearing in Theorem \ref{harnack} reveals that these constants are very small. We will therefore use the (crude) upper bounds $\eps_1 \leq \frac{1}{4}$ and $\eta\leq \frac{1}{5}$ to simplify some our estimates. In particular, these bounds imply that $-\log_2(1-\eta) \leq \frac{1}{4}$.
\end{remark}

Finally, we will make use of the following estimate later in the proof. While it's proof is standard and can be found for instance in \cite[Proposition 4.13]{cafcabtextbook}, we restate it here as we need the constant appearing in the estimate \eqref{eqn:boundary:regularity}.
\begin{proposition}\label{prop:holder}
	Suppose that $u \in C^0(B_1)$ satisfies 
	\begin{equation*}
		\begin{cases}
			\Delta u = 0 &\text{ in } B_1\\
			u=g &\text{on } \p B_1,
		\end{cases}
	\end{equation*}
	for some $ \varphi \in C^{\sigma} (\p B_1)$. Then $u \in C^{\frac{\sigma}{2}}(\overline{B}_1)$ with 
	\begin{equation}\label{eqn:boundary:regularity}
		\norm{u}_{C^{0,\frac{\sigma}{2}(\overline{B}_1)}}\leq 2n 5^{\sigma} \norm{g}_{C^{0,\sigma}(\p B_1)}.
	\end{equation}
\end{proposition}

\section{Proof of Theorem \ref{improvementofflatness}}

\begin{lemma}\label{stretchgraph}
	Suppose $\p E$ is a minimal surface with $\p E \cap B_1 \subset \{\abs{x_n} \leq \eps\}$ for some $0 < \eps \leq \frac{1}{8}\eps_1(n)$ and that $0 \in \p E$.There exists a multivalued graph $A:B_1'\to 2^{\R}$ such that $\abs{A}\leq 1$, $0\in A(0)$ and $\p E = \eps A$. Moreover there holds for any $x',y' \in B_{3/4}'$ that
	\begin{equation}
		\label{almostholder}
		\abs{A(x')-A(y')}\leq C_1(\abs{x'-y'}+C_2\eps^{\gamma})^{\alpha}
	\end{equation}
	where $\alpha = -\log_2(1-\eta)$, $\gamma = \frac{1}{1-\alpha}$, $C_1=2^{4+4\alpha}$ and $C_2 = 2^{\frac{3\alpha}{1-\alpha}-1}\eps_1^{-\gamma}$ are all universal constants.
\end{lemma}
\begin{proof}
	The first part of the lemma is easily seen by setting 
	\begin{equation*}
		A(x') = \left\{\frac{x_n}{\eps}; \hspace{1mm} (x',x_n)\in \p E\right\}.
	\end{equation*}
	To deduce \eqref{almostholder} we iteratively apply Theorem  \ref{harnack} to each point in $\p E \cap B_{3/4}$. To this end, we fix some $x_0 \in B_{3/4}$ and note that by assumption
	\begin{equation*}
		\abs{\left(x-x_0\right)\cdot e_n}\leq 2\eps,  \text{	for all  $x\in \p E \cap B_{1/4}(x_0)$.}
	\end{equation*}
	 Since $8\eps \leq \eps_1$ we can apply Theorem \ref{harnack} in $B_{1/4}(x_0)$ and obtain 
	\begin{equation*}
		\p E \cap B_{1/8}(x_0) \subset \left\{\abs{\left(x-x_0\right)\cdot e_n}\leq 2\eps (1-\eta)\right\}.
	\end{equation*}
	 Iterating this procedure, we obtain for all $m\geq 3$ satisfying
	\begin{equation}
		\label{mcondition}
		2^m\eps(1-\eta)^{m-3}\leq \eps_1
	\end{equation}
	that 
	\begin{equation}
		\label{oscboundarydecay}
		\p E \cap B_{2^{-m}}(x_0) \subset \left\{\abs{\left(x-x_0\right)\cdot e_n}\leq 2\eps (1-\eta)^{m-2}\right\}.
	\end{equation}
	Using now \eqref{mcondition} and Remark \ref{remark:assumption} we have that 
	\begin{align*}
		2\eps(1-\eta)^{m-2} = \eps_1^{-1}\eps(1-\eta)^{m-3}(2\eps_1(1-\eta)) \leq 2^{-m-1},
	\end{align*}
	and so
	\begin{equation*}
		\p E \cap \left\{\abs{x'-x_0'}\leq 2^{-1-m}\right\}  \subset B_{2^{-m}}(x_0).
	\end{equation*}
	Hence as a consequence of \eqref{oscboundarydecay} we have that
	\begin{equation*}
		A \cap \left\{\abs{x'-x_0'}\leq 2^{-1-m}\right\} \subset \left\{\abs{\left(x-x_0\right)\cdot e_n}\leq 2 (1-\eta)^{m-2}\right\},
	\end{equation*}
	which implies
	\begin{equation}
		\underset{{B'_{2^{-1-m}}(x_0')}}{\osc} A \leq 2(1-\eta)^{m-2}.
		\label{oscdecay}
	\end{equation}
	Moreover using \eqref{mcondition} we deduce that the oscillation decay of \eqref{oscdecay} holds for $m\leq \lfloor M(\eps,n) \rfloor =: \tilde{M}$ where  $M(\eps,n)$ achieves the equality in \eqref{mcondition}. We note that with this maximal value of $m$, we have that $2^{-\tilde{M}-1} \leq 2^{\frac{3\alpha}{1-\alpha}-1}\eps_1^{-\gamma}\eps^{\gamma} =C_2\eps^{\gamma}$.
	
	We now show that this implies that $A$ is H\"older continuous at $x_0'$ outside of $B_{C_2\eps^{\gamma}}(x_0')$. Fix any $y' \in B_{1/4}(x_0')\backslash B_{C_2\eps^{\gamma}}(x_0')$ and choose $m\in \N$ so that
	\begin{equation*}
		2^{-m-1} \leq \abs{x_0'-y'} \leq 2^{-m}.
	\end{equation*}
	Since $m \leq \tilde{M}$ we have by \eqref{oscdecay} that 
	\begin{align*}
		\abs{A(x_0')-A(y')}
		&\leq \underset{B'_{2^{-m}(x_0')}}{\osc} A\\
		&\leq 2(1-\eta)^{m-3}\\
		&\leq 2^{-\alpha(m-3)+1}\\
		&\leq \frac{C_1}{6} \abs{x_0'-y'}^{\alpha}.
	\end{align*}
	If on the other hand $y'\in B_{C_2\eps^{\gamma}}(x_0')$ we note that $2^{-\tilde{M}-1}\leq \abs{x'-y'} + C_2\eps^{\gamma}$ and so using the oscillation decay in $B_{2^{-\tilde{M}}}(x_0')$ we obtain
	\begin{equation*}
		\abs{A(x_0')-A(y')}\leq \frac{C_1}{6}\left(\abs{x'-y'}+C_2\eps^{\gamma}\right)^{\alpha}.
	\end{equation*}
	In particular we have shown \eqref{almostholder} for all $y'\in B'_{1/4}(x_0')$. Taking any $y' \in B'_{3/4}\backslash B'_{1/4}(x_0')$ we can find a sequence of $I\leq 5 $ points $\{x'_i\}_{1\leq i\leq I} \subset B'_{3/4}$ such that $x_I'=y'$ and $\abs{x'_i - x_{i+1}'}\leq \frac{1}{4}$ for $0\leq i \leq 5$. Applying the triangle inequality then establishes \eqref{almostholder}.
\end{proof}
\begin{remark}\label{remark:value:alpha}
	By Remark \ref{remark:assumption}, the constants $\gamma$ and $C_2$ are well defined. Moreover, using the bound $\alpha \leq \frac{1}{4}$ we find that $C_2\leq\eps_1^{-\gamma}$.
\end{remark}

Recalling that $\p E$ is closed, we now define the lower semi-continuous function
\begin{equation*}
	u^{-}(x') =\eps \inf A(x'),
\end{equation*}
as well as the upper semi-continuous function
\begin{equation*}
	u^+(x') = \eps \sup A(x').
\end{equation*}
Geometrically, these correspond to the lower and upper parts of the minimal surface. Since $0 \in \p E$ we can assume that $u^{-}(0) = 0$ (otherwise we could translate the minimal surface and complete the proof under the assumption that $\p E \cap B_1 \subset \{\abs{x_n}\leq 2\eps\}$).

\begin{lemma}\label{holdergraph}
	Under the assumptions of Lemma \ref{stretchgraph} there exists some $u \in C^{0,\alpha}(B_{3/4}')$ with $[u]_{C^{0,\alpha}(B_{3/4}')}\leq 2^{4+5\alpha}$ and $u(0)=0$, such that
	\begin{equation}
		\label{sandwhich}
		\eps u(x') \leq u^-(x') \leq u^+(x')\leq \eps u(x')+ C_3\eps^{1+\gamma\alpha}
	\end{equation}
	where $C_3 = 2^{4+5\alpha}\eps_1^{-\gamma\alpha}$. In particular we have that
	\begin{equation}
		\label{eqn:difference:A:u}
		\norm{A-u}_{L^{\infty}(B_{3/4}')} \leq C_3\eps^{\gamma\alpha},
	\end{equation}
	and
	\begin{equation}\label{eqn:holderbound}
		\norm{u}_{C^{0,\alpha}(\overline{B'}_{1/2})}\leq 2^{6+5\alpha}.
	\end{equation}
\end{lemma}
\begin{proof}
	We define the function
	\begin{equation*}
		u(x_0') = \inf_{x'\in B'_{3/4}}\left\{\frac{u^-(x')}{\eps} + 2^{\alpha}C_1\abs{x'-x_0'}^{\alpha}\right\}, \hspace{2mm} x_0' \in B'_{3/4}.
	\end{equation*}
	To see that $u \in C^{0,\alpha}$ take $x_0', x_1' \in B'_{3/4}$ and let $x_0'^*\in \overline{B'}_{3/4}$ such that $u(x_0')=\frac{u^-(x_0'^*)}{\eps} + 2^{\alpha}C_1\abs{x_0'^*-x_0'}^{\alpha}$. Then we observe that
	\begin{align*}
		u(x_1') - u(x_0')
		&\leq \left(\frac{u^-(x_0'^*)}{\eps}+ 2^{\alpha}C_1\abs{x_0'^*-x_1'}^{\alpha}\right)-\left(\frac{u^-(x_0'^*)}{\eps} + 2^{\alpha}C_1\abs{x_0'^*-x_0'}^{\alpha}\right)\\
		&\leq 2^{\alpha}C_1\left(\abs{x_0'^*-x_1'}^{\alpha}-\abs{x_0'^*-x_0'}^{\alpha}\right)\\
		&\leq 2^{\alpha}C_1\abs{x_0'-x_1'}^{\alpha}.
	\end{align*}
	A symmetrical argument yields  $u(x_0') - u(x_1')\leq 2^{\alpha}C_1\abs{x_0'-x_1'}^{\alpha}$ which establishes $[u]_{C^{0,\alpha}(B_{3/4}')}\leq 2^{4+5\alpha}$. The fact that $u(0)=0$ follows immediately from $u^-(0)=0$ and the definition of $u$. 
	
	The lower bound in \eqref{sandwhich} follows directly from the definition of $u$ and multiplying by $\eps$. For the upper bound we first observe that \eqref{almostholder} yields
	\begin{equation*}
		\frac{u^-(x')}{\eps} \geq \frac{u^+(x_0')}{\eps} - C_1(\abs{x'-x_0'}+ C_2\eps^{\gamma})^{\alpha}
	\end{equation*}
	for any $x_0',x'\in B'_{3/4}$. Consequently we have that 
	\begin{align*}
		\frac{	u^-(x')}{\eps}+ 2^{\alpha}C_1\abs{x'-x_0'}^{\alpha}
		&\geq \frac{u^+(x_0')}{\eps} - C_1(\abs{x'-x_0'}+ C_2\eps^{\gamma})^{\alpha} + 2^{\alpha}C_1\abs{x'-x_0'}^{\alpha}\\
		&\geq  \frac{u^+(x_0')}{\eps} - 2^{\alpha}C_1C_2^{\alpha}\eps^{\gamma\alpha}.
	\end{align*}
	Taking now the infimum and multiplying by $\eps$ yields the upper bound in \eqref{sandwhich}. By subtracting $\eps u$ in \eqref{sandwhich} and then dividing by $\eps$ we obtain \eqref{eqn:difference:A:u}. 
	
	The estimate \eqref{eqn:holderbound} follows from \eqref{eqn:difference:A:u}, the fact that $\abs{A} \leq 1$ and the estimate $[u]_{C^{0,\alpha}(B_{3/4}')}\leq 2^{4+5\alpha}$. Indeed, since $\eps \leq \eps_1$, we have that
	\begin{align*}
		\norm{u}_{C^{0,\alpha}(\overline{B'}_{1/2})} 
		&\leq \norm{u}_{L^{\infty}(B_{3/4}')}+ [u]_{C^{0,\alpha}(B_{3/4}')}\\
		&\leq (1 + C_3\eps^{\gamma\alpha}) + 2^{4+5\alpha}\\
		&\leq 2^{4+5\alpha}(2^{-4-5\alpha} + \eps_1^{-\gamma\alpha}\eps^{\gamma\alpha}) + 2^{4+5\alpha}\\
		&\leq 2^{6+5\alpha}.
	\end{align*}
\end{proof}
We must now understand what properties of $u$ we can deduce from the minimal surface equation that $\p E$ satisfies. To this end, we define the universal constant $r=\eps^{\frac{\gamma\alpha}{4}}$ and we note that $\frac{\gamma\alpha}{4} \leq \frac{1}{12}$.

\begin{lemma}\label{almostviscosity}
	Suppose $\varphi \in C^{\infty}(B_{1/2}')$  touches $u+ C_3 \eps^{\gamma\alpha}$ from above at some $x_0' \in B'_{1/2-2r}$. If $\max\{\sup_{i,j} \abs{D_{ij}\varphi}, \abs{D\varphi},\abs{D^2\varphi}\} \leq \eps^{-\frac{1}{2}}$ in $B'_{1/2-r}$ for some $\eps \leq\left(2^{-6-5\alpha} \eps_1^{\gamma\alpha} \right)^{\frac{2}{1+\gamma\alpha}}$, then there exits $x'_1\in B'_{r}(x'_0)$ so that
	\begin{equation}
		\Delta \varphi (x'_1) \geq - C_4\eps^{\frac{\gamma\alpha}{2}},
		\label{viscosityequation}
	\end{equation}
	where $C_4 = 2^{8+5\alpha}n\eps_1^{-\gamma\alpha}$.
\end{lemma}
\begin{proof}
	We add a sufficiently large paraboloid $P$ to $\eps \varphi$ so that a translation of $\eps \varphi + P$ touches $u^+$ from above in $B'_r(x'_0)$. That is, for some $\delta >0$ to be determined there exists some $x'_1\in B'_r(x'_0)$ such that 
	\begin{align}
		\label{minimisingdistance}
		\eps \varphi(x'_1) + \frac{\delta}{2} \abs{x_1'-x_0'}^2-u^+(x_1') 
		= \min_{\overline{B'_r(x_0')}}\left\{\eps \varphi + \frac{\delta}{2} \abs{x'-x_0'}^2-u^+(x')\right\}.
	\end{align}
	Now if $x_1' \in \p B'_r(x_0')$ we have by \eqref{minimisingdistance} that
	\begin{align*}
		\eps \varphi(x'_1) + \frac{\delta}{2} r^2 -u^+(x'_1)
		&\leq 	\eps \varphi(x_0')  - u^+(x_0'),
	\end{align*}
	and since $	\eps \varphi(x_1')  - u^+(x_1') \geq 0$ we obtain
	\begin{align*}
		\eps \varphi(x_0') \geq  \frac{\delta}{2} r^2 +u^+(x_0').
	\end{align*}
	We then find
	\begin{align*}
		0= \eps u(x_0') + C_3\eps^{1+\gamma\alpha}-\eps \varphi(x_0') 
		&\leq \eps u(x_0') + C_3\eps^{1+\gamma\alpha} - \frac{\delta}{2} r^2 - u^+(x_0')\\
		&\leq C_3\eps^{1+\gamma\alpha} - \frac{\delta}{2} r^2,
	\end{align*}
	which is a contradiction if 
	\begin{equation*}
		\delta = 4C_3\eps^{1+\frac{\gamma\alpha}{2}}.
	\end{equation*}
	Now with this $\delta$ we have that $\eps \varphi(x')+\frac{\delta}{2}\abs{x'-x_0'}^2$ is tangent from above to $u^+$ at $x_1'$ in $B'_r(x_0')$. Since $\p E$ is a viscosity solution of \eqref{MSE} we have that 
	\begin{align*}
		&\left(1+\abs{\eps D\varphi(x_1') + \delta (x_1'-x_0')}^2\right)\left(\eps \Delta \varphi(x_1') + (n-1)\delta\right) \\
		&+  \left(\eps D\varphi(x_1') + \delta (x_1'-x_0')\right)^T\left(\eps D^2\varphi(x_1')+\delta I \right) \left(\eps D\varphi(x_1') + \delta (x_1'-x_0')\right) \geq 0.
	\end{align*}
	Using the bounds on $\abs{D\varphi}$ and $\abs{D^2\varphi}$, along with the fact that $\abs{x_0'-x_1'}\leq r$, we find 
	\begin{align*}
		\abs{\eps D\varphi(x_1') + \delta (x_1'-x_0')}
		&\leq \eps^{\frac{1}{2}} + 4C_3 \eps^{1+\frac{3\gamma\alpha}{4}} \\
		&\leq \eps^{\frac{1}{2}}\left(1 + 2^{6+5\alpha}\eps_1^{-\gamma\alpha}\eps^{\frac12+\frac34\gamma\alpha}\right) \\
		&\leq 2\eps^{\frac{1}{2}}
	\end{align*}
	for $\eps \leq (2^{-6-5\alpha}\eps_1^{\gamma\alpha} )^{\frac{4}{2+3\gamma\alpha}}$. A similar computation yields the bound 
	\begin{equation*}
		\abs{\eps D^2\varphi(x_1') +\delta I } \leq 2\eps^{\frac{1}{2}}
	\end{equation*}
	for $\eps \leq (2^{-6-5\alpha}\eps_1^{\gamma\alpha} )^{\frac{2}{1+\gamma\alpha}} \leq (2^{-6-5\alpha}\eps_1^{\gamma\alpha} )^{\frac{4}{2+3\gamma\alpha}}$. Using these bounds we see that $(1+\abs{\eps D\varphi(x_1') + \delta (x_1'-x_0')}^2)(n-1)\delta \leq 2n\delta$ and so we find that
	\begin{align*}
		\eps \Delta \varphi(x_1') 
		&\geq -\eps\abs{\Delta\varphi(x_1')}\abs{\eps D\varphi(x_1') + \delta (x_1'-x_0')}^2 \\&\hspace{3mm} - (1+\abs{\eps D\varphi(x_1') + \delta (x_1'-x_0')}^2 )(n-1)\delta \\&\hspace{3mm}- \abs{\eps D\varphi(x_1') + \delta (x_1'-x_0')}^2\abs{\eps D^2\varphi(x_1') +\delta I }\\
		&\geq - 12n\eps^{3/2} - 2n\delta \\
		&\geq -C_4\eps^{1+\frac{\gamma\alpha}{2}},
	\end{align*}	
	where in the last inequality we have used that $\frac{\gamma\alpha}{2}\leq \frac12$.
\end{proof}

We will now construct a function that is strictly above $u$ for $\eps$ small enough. Recalling that $u \in C^{0,\alpha}(\overline{B}'_{1/2})$, there exists $v \in C^{0,\alpha/2}(\overline{B}'_{1/2})$ satisfying
\begin{equation*}
	\begin{cases}
		\Delta v =  - 2C_4\eps^{\frac{\gamma\alpha}{2}} & \text{in } B_{1/2}'\\
		v = u & \text{on } \p B_{1/2}'
	\end{cases}.
\end{equation*}
Moreover, splitting $v=v_1+v_2$ where 
\begin{equation*}
	\begin{cases}
		\Delta v_1 =  0 & \text{in } B_{1/2}'\\
		v_1 = u & \text{on } \p B_{1/2}',
	\end{cases}
\end{equation*}
and $v_2 = - C_4\eps^{\frac{\gamma\alpha}{2}}\frac{1}{n-1}(\abs{x'}^2-\frac14)$, we have using Proposition \ref{prop:holder}, \eqref{eqn:holderbound} and $\eps\leq \eps_1$ that
\begin{align*}
	\norm{v}_{C^{0,\alpha/2}(\overline{B}'_{1/2})}
	 &\leq 2n\cdot5^{\alpha}\norm{u}_{C^{0,\alpha}(\overline{B}'_{1/2})} + 2C_4\eps^{\frac{\gamma\alpha}{2}}\\
	 &\leq 2n\cdot(2^3)^{\alpha}(2^{6+5\alpha})+2C_4\eps^{\frac{\gamma\alpha}{2}}\\
	 & \leq 2^{10+5\alpha}n\eps_1^{-\frac{\gamma\alpha}{2}} =:C_5.
\end{align*}
\begin{lemma}\label{constructwplus}
	Let $w_+\in C^{0,\frac{\alpha}{2}}(\overline{B}_{1/2}')\cap C^{\infty}(B_{1/2}')$ satisfy
	\begin{equation}
		\label{wplus}
		\begin{cases}
			\Delta w_+ =  - 2C_4\eps^{\frac{\gamma\alpha}{2}} & \text{in } B_{1/2}'\\
			w_+ = u + 4C_5r^{\frac{\alpha}{2}} & \text{on } \p B_{1/2}'
		\end{cases}.
	\end{equation}
	For all $x_0'\in B'_{\frac{1}{2}-r}$ we have that
	\begin{equation}\label{eqn:wplus:derivatives}
		\max\left\{\sup_{i,j} \abs{D_{ij}w_+(x_0')}, \abs{Dw_+(x_0')},\abs{D^2w_+(x_0')}\right\} \leq \eps^{-\frac{1}{2}}
	\end{equation}
	as long as $\eps \leq \left(2^{-18-5\alpha}n^{-5}\eps_1^{\gamma\alpha}\right)^{\frac{2}{1-\gamma\alpha}}$.
	Consequently, $w_+ > u$ in $B_{1/2}'$.
\end{lemma}
\begin{proof}
	We split $w_+ = (w_+)_1 + (w_+)_2$ where
	\begin{equation*}
		\begin{cases}
			\Delta (w_+)_1 =  0 & \text{in } B_{1/2}'\\
			(w_+)_1  = u + 4C_5r^{\frac{\alpha}{2}} & \text{on } \p B_{1/2}'
		\end{cases}
	\end{equation*}
	and $(w_+)_2 = - C_4\eps^{\frac{\gamma\alpha}{2}}\frac{1}{n-1}(\abs{x'}^2-\frac14)$. A direct computation shows that
	\begin{equation*}
		\sup_{1\leq i\leq n}\abs{D_i(w_+)_2} \leq \sup_{1\leq i,j \leq n} \abs{D_{ij}(w_+)_2} \leq \frac1{2n^2}\eps^{-\frac12}
	\end{equation*}
	for $\eps \leq \left(2^{-10-5\alpha}n^{-2}\eps_1^{\gamma\alpha}\right)^{\frac{2}{1+\gamma\alpha}}$. It then immediately follows that
	\begin{equation}
		\label{eqn:wplus2:derivatives}
		\sup_{1\leq k \leq 2} \abs{D^k(w_+)_2} \leq \frac12\eps^{-\frac12} \text{ in } B'_{1/2}.
	\end{equation}
	Moreover, for any $x_0' \in B'_{1/2-r}$ we can apply the derivative estimates for harmonic functions in $B'_{\frac{r}{2}}(x_0')\subset B'_{1/2}$ and obtain for any $1\leq i \leq n$ that
	\begin{align*}
		\abs{D_i(w_+)_1(x_0')} \leq \frac{2n}{r} \norm{(w_+)_1}_{L^{\infty}(B'_{1/2})} \leq 2n\eps^{-\frac{\gamma\alpha}{4}} ( \norm{u}_{L^{\infty}(B'_{1/2})}+4C_5r^{\frac{\alpha}{2}}),
	\end{align*}
	where in the last inequality we have used the maximum principle. Similarly for any $1\leq i,j \leq n$ we find that
	\begin{align*}
		\abs{D_{i,j}(w_+)_1(x_0')} \leq \left(\frac{4n}{r}\right)^2 \norm{(w_+)_1}_{L^{\infty}(B'_{1/2})}\leq 16n^2\eps^{-\frac{\gamma\alpha}{2}} ( \norm{u}_{L^{\infty}(B'_{1/2})}+4C_5r^{\frac{\alpha}{2}}).
	\end{align*}
	Choosing $\eps \leq \left(2^{-18-5\alpha}n^{-5}\eps_1^{\gamma\alpha}\right)^{\frac{2}{1-\gamma\alpha}} \leq \left(2^{-10-5\alpha}n^{-1}\eps_1^{\gamma\alpha}\right)^{\frac{2}{1+\gamma\alpha}}$ we see that
	\begin{equation*}
		\max\left\{\sup_{1\leq i\leq n}\abs{D_i(w_+)_1(x_0')}, \sup_{1\leq i,j \leq n} \abs{D_{ij}(w_+)_1(x_0')} \right\} \leq \frac1{2n^2}\eps^{-\frac12}.
	\end{equation*}
	Consequently we obtain that
	\begin{equation*}
		\sup_{1\leq k \leq 2} \abs{D^k(w_+)_1} \leq \frac12\eps^{-\frac12} \text{ in } B'_{1/2-r}.
	\end{equation*}
	This coupled with \eqref{eqn:wplus2:derivatives} establishes \eqref{eqn:wplus:derivatives}.\\
	Since  $\eps \leq \left(2^{-18-5\alpha}n^{-5}\eps_1^{\gamma\alpha}\right)^{\frac{2}{1-\gamma\alpha}} \leq \left(2^{-6-5\alpha} \eps_1^{\gamma\alpha} \right)^{\frac{2}{1+\gamma\alpha}}$, $w_+$ will not touch $u$ in $B'_{1/2-2r}$ or else that would contradict Lemma \ref{almostviscosity}. Moreover $w_+$ will not touch $u$ in the annulus $B'_{1/2}\backslash B'_{1/2-2r}$ or else that would violate the H\"older continuity of both $u$ and $w_+$.
\end{proof}
Arguing similarly as in Lemma \ref{almostviscosity} and Lemma \ref{constructwplus}, we can construct a function $w_-$ satisfying
\begin{equation}
	\label{wminus}
	\begin{cases}
		\Delta w_- =   2C_4\eps^{\frac{\gamma\alpha}{2}} & \text{in } B_{1/2}'\\
		w_- = u - 4C_5r^{\frac{\alpha}{2}} & \text{on } \p B_{1/2}',
	\end{cases}
\end{equation}
and which lives below the graph of $u$ in $B'_{1/2}$ for $\eps \leq \left(2^{-18-5\alpha}n^{-5}\eps_1^{\gamma\alpha}\right)^{\frac{2}{1-\gamma\alpha}}$.
The existence of $w_+$ and $w_-$ now allows us to show that $u$ is actually very close to it's harmonic replacement. 

\begin{proposition}
	\label{closetoharmonic}
	Let $w\in C^{0,\frac{\alpha}{2}}(\overline{B}_{1/2}')\cap C^{\infty}(B_{1/2}')$ be the solution of
	\begin{equation*}
		\begin{cases}
			\Delta w = 0 & \text{in } B_{1/2}'\\
			w = u  & \text{on } \p B_{1/2}'.
		\end{cases}
	\end{equation*}
	If $\eps \leq \left(2^{-18-5\alpha}n^{-5}\eps_1^{\gamma\alpha}\right)^{\frac{2}{1-\gamma\alpha}}$ then $\norm{u-w}_{L^{\infty}(B'_{1/2})} \leq C_6\eps^{\frac{\gamma\alpha^2}{8}}$ where $C_6 = 2^{14+5\alpha}n\eps_1^{-\gamma\alpha}$.
\end{proposition}
\begin{proof}
	Using \eqref{wplus} and \eqref{wminus} we have that
	\begin{equation*}
		\begin{cases}
			\Delta (w_+ - w_-)=  -4C_4\eps^{\frac{\gamma\alpha}{2}} & \text{in } B_{1/2}'\\
			w_+ - w_- =  8C_5r^{\frac{\alpha}{2}} & \text{on } \p B_{1/2}'
		\end{cases}
	\end{equation*}
	and so by the maximum principle we have
	\begin{equation*}
		\norm{w_+-w_-}_{L^{\infty}(B_{1/2}')} \leq 8C_5\eps^{\frac{\gamma\alpha^2}{8}} + \frac12\frac{C_4}{n-1}\eps^{\frac{\gamma\alpha}{2}} \leq C_6\eps^{\frac{\gamma\alpha^2}{8}}.
	\end{equation*}
	Since for $\eps \leq \left(2^{-18-5\alpha}n^{-5}\eps_1^{\gamma\alpha}\right)^{\frac{2}{1-\gamma\alpha}}$ we have $w_- < u < w_+$, and since  $w_- < w < w_+$, the result follows.
\end{proof}
We can now finally give the
\begin{proof}[Proof of Theorem \ref{improvementofflatness}]
	We assume that $\eps \leq \left(2^{-18-5\alpha}n^{-5}\eps_1^{\gamma\alpha}\right)^{\frac{2}{1-\gamma\alpha}}$ so that Proposition \ref{closetoharmonic} holds. Moreover, in order to simplify our estimates we will use the bound $\alpha \leq \frac14$ (c.f. Remark \ref{remark:assumption}) so that in particular $\norm{u}_{L^{\infty}(\overline{B}'_{1/2})} \leq 2^6$. \\
	Since $w$ is harmonic, we have for any $x'\in B'_{2r\naught}(0)$ with $r\naught\leq \frac{1}{8}$ that
	\begin{align*}
		\abs{w(x')-w(0)-\nabla w(0)\cdot x'}
		&\leq 2r\naught^2 \max_{1\leq i,j\leq n}\norm{D_{ij}w}_{L^{\infty}(B'_{\frac14})}\\
		&\leq 128n^2r\naught^2\norm{w}_{L^{\infty}(\p{B}'_{1/2})}\\
		&\leq 128n^2r\naught^2\norm{u}_{L^{\infty}(\overline{B}'_{1/2})}\\
		&\leq 2^{13}n^2r\naught^2.
	\end{align*}
	Choosing $r\naught= 2^{-16}n^{-2}$ we have that
	\begin{equation}
		\abs{w(x')-w(0)-\nabla w(0)\cdot x'} \leq \frac{r\naught}{8}.
		\label{linearpparoximation}
	\end{equation}
	Since $u(0) =0$ we have using Proposition \ref{closetoharmonic} for any $x'\in B'_{2r\naught}(0)$ that
	\begin{align*}
		\abs{u(x')-\nabla w(0)\cdot x'}
		&\leq \abs{u(x')-w(x')}+\abs{u(0)-w(0)}+\abs{w(x')-w(0)-\nabla w(0)\cdot x'}\\
		&\leq 2C_6\eps^{\frac{\gamma\alpha^2}{8}} + \frac{r\naught}{8}\\
		&\leq \frac{3r\naught}{8}
	\end{align*}
	by choosing $\eps \leq \left(\frac{r\naught}{8C_6}\right)^{\frac{8}{\gamma\alpha^2}}$. Since with this choice of $\eps$ we have that $\eps \leq \left(\frac{r\naught}{8C_3}\right)^{\frac{1}{\gamma\alpha}}$, we obtain using \eqref{eqn:difference:A:u} that
	\begin{equation*}
		A \cap \left\{\abs{x'}\leq 2r\naught\right\}\subset \left\{\abs{x_n - \nabla w(0)\cdot x'}\leq \frac{r\naught}{2}\right\}.
	\end{equation*}
	It then immediately follows that
	\begin{equation*}
		\p E \cap B_{r_0} \subset \left\{\abs{x\cdot \nu}\leq\frac{\eps}{2}r_0\right\},
	\end{equation*}
	where $\nu = \frac{(-\eps \nabla w(0),1)}{\abs{(-\eps \nabla w(0),1)}}$.  Since $\alpha \leq \frac14$ we also readily have that $\left(\frac{r\naught}{8C_6}\right)^{\frac{8}{\gamma\alpha^2}}\leq  \left(2^{-18-5\alpha}n^{-5}\eps_1^{\gamma\alpha}\right)^{\frac{2}{1-\gamma\alpha}}$ and so we can take $\eps\naught =\left(\frac{r\naught}{8C_6}\right)^{\frac{8}{\gamma\alpha^2}}$ which concludes the proof. 
\end{proof}

\renewcommand{\bibname}{References}
\bibliographystyle{plain}
\bibliography{references}
\end{document}